\documentclass[matzei,envcountsame,draft]{svjour}
	
	\def\makeheadbox{{%
	\hbox to0pt{\vbox{\baselineskip=10dd\hrule\hbox
	to\hsize{\vrule\kern3pt\vbox{\kern3pt
	\hbox{submitted version, compiled in Frankfurt am Main, November 21, 2009}
	\kern3pt}\hfil\kern3pt\vrule}\hrule}%
	\hss}}}
\usepackage[latin1]{inputenc}%
\usepackage{amsmath,amssymb,latexsym}%
\def\pf#1{{\def\temp{#1} 
              \ifx\temp\empty 
                  \noindent\slshape\textbf{proof:\ }
              \else 
                  \noindent\slshape\textbf{proof of\ #1:\ }
              \fi}}
\def\qed{\hspace*{\fill}$\square$}
\def\Fix#1_#2.{{#2_{#1}}}%
\def\bs#1.{
              \def\temp{#1} 
              \ifx\temp\empty 
                   \mathcal{B}
              \else
                   \mathcal{B}(#1)
              \fi
}

\def\Aut#1_#2.{\mathsf{Aut}_{#2}(#1)}%

\renewcommand{\emph}{\textbf}%
\newcommand{\flatrk}{flat-rank\,}

\begin{document}
\title{
Hyperbolic groups have flat-rank at most~$1$}
\author{Udo Baumgartner, 
R{\"o}gnvaldur G. M{\"o}ller\inst{2}  
\and 
George A. Willis\inst{1}
\thanks{\emph{The first and last author were supported by A.R.C. Grants DP0208137 and LX0349209}}%
}                     
%
\institute{School of Mathematical and Physical Sciences\\
                 University of Newcastle\\
                 University Drive\\
                 Callaghan, NSW 2308\\
                 Australia\\
                 \email {George.Willis@newcastle.edu.au}
                  \and
                 Science Institute\\
                 University of Iceland\\
                 107 Reykjavik\\
                  Iceland\\
                  \email{roggi@raunvis.hi.is}}
\date{}
%
\authorrunning{Baumgartner, M\"oller and Willis}
\maketitle

\begin{abstract}
The \flatrk of a totally disconnected, locally compact group $G$ 
is an integer, 
which is 
an invariant of $G$ 
as a topological group. 
We generalize 
the concept of 
hyperbolic groups 
to the topological context 
and 
show that 
a totally disconnected, locally compact, hyperbolic group 
has \flatrk at most~$1$. 
It follows that 
the simple totally disconnected locally compact groups 
constructed by Paulin and Haglund 
have \flatrk at most~$1$. 
\end{abstract}


\section{Introduction}\label{sec:intro}

\label{section-Intro}
The concept 
of a hyperbolic group 
can be generalized 
to the realm of 
compactly generated, topological groups 
by 
a straightforward 
adaption of 
the definition 
in the discrete case 
(see Definition~\ref{def:hyp-topG} 
on page~\pageref{def:hyp-topG}).
Such a generalization 
is an instance of 
\textit{`geometric group theory 
for topological groups'\/}, which is  
a line 
of investigation 
proposed in~\cite{analog(CayleyGphs(topGs))}.  

This geometric approach 
is a natural one in the case of totally disconnected, locally compact groups, the subject of this paper, and has been pursued previously in 
\cite{struc(tdlcG-graphs+permutations)}, 
\cite{Gphs+Perm+topGs}, 
\cite{tdlcGs-geomObjs}
and \cite{geo-char(flatG(auts))}. 
We take a different line to these papers, however, by studying hyperbolicity and relating it to a structural invariant for totally disconnected, locally compact groups, namely, the \emph{flat-rank}. 

The \flatrk\!  of a totally disconnected, locally compact group (see Definition~\ref{defn:flat-rank}) is a non-negative integer that is analogous to the $k$-rank of a semisimple algebraic group over a local field~$k$.  Indeed the \flatrk\!  and  $k$-rank
coincide when  $G$ is such 
a group, 
by Corollary~19 
in~\cite{flatrk(AutGs(buildings))}.

Just as the $k$-rank 
of a simple algebraic group 
determines 
many properties of the group, the \flatrk\! can be expected to convey important information about  general simple totally disconnected groups. 
An indication of this is seen with the computation of 
the \flatrk\! 
of automorphism groups 
of buildings made 
in~\cite{flatrk(AutGs(buildings))}, where it is shown, in conjunction with 
results from~\cite{geoFlats<CAT0-real(CoxGs+Tits-buildings)}, that if the group action is sufficiently transitive then 
the \flatrk\! 
of the group 
equals 
the rank 
of its building. 
The following theorem, our main result, further demonstrates the relationship between the \flatrk\! and geometric properties of the group. 

\begin{theorem}\label{MainThm-variant}
The \flatrk\! of 
a totally disconnected, locally compact, hyperbolic group 
is at most~$1$. 
\end{theorem}

The major part 
of this paper 
is devoted to 
the proof of 
Theorem~\ref{MainThm-variant}. 
It is clear that the converse to this theorem does not hold, because discrete groups have \flatrk\!~0 and need not be hyperbolic. However, it may hold in the presence of further hypotheses that exclude discrete groups or non-discrete counterexamples based on them.

The properties of algebraic groups 
of $k$-rank $1$ 
differ notably 
from the properties 
of groups of higher $k$-rank. In the expectation that the same will be true of the \flatrk\!,   a secondary aim of this paper is to seek further geometric
criteria 
for a totally disconnected, locally compact group 
to have 
\flatrk\! 
at  most~$1$. We establish two such criteria. 

One is 
based on the action of the group on
the space of compact open subgroups 
of the group.  
The criterion 
and its proof 
are
in the spirit 
of
the papers~\cite{flatrk(AutGs(buildings))} 
and~\cite{geo-char(flatG(auts))}; 
the proof 
is contained 
in Section~\ref{sec:hyperbolic orbits}.

\begin{theorem}\label{thm:altMainThm}
Let $\mathcal{A}$ be 
a group of automorphisms 
of the totally disconnected locally compact group $G$. 
Suppose that $\mathcal{A}$ has a hyperbolic orbit 
in the space of compact open subgroups of $G$. 
Then the \flatrk\! of $\mathcal{A}$ 
is at most $1$. 
\end{theorem}

The last criterion 
follows from 
results in~\cite{direction(aut(tdlcG))}, where the space of directions of a totally disconnected, locally compact group is defined; 
see 
page~\pageref{pref:pf(thm:flatrk-via-directions)} 
for 
its proof.  

\begin{theorem}\label{thm:flatrk-via-directions}
Let~$G$ be 
a totally disconnected, locally compact group 
whose space of directions 
is discrete. 
Then 
the \flatrk\! 
of~$G$ 
is at most~$1$. If the space of directions is not empty, then the \flatrk\! is exactly~1.
\end{theorem} 

We do not know of a  
hyperbolic group whose space of directions is not discrete. In view of Theorem~\ref{thm:flatrk-via-directions}, it would be a strengthening of Theorem~\ref{MainThm-variant} to show that all hyperbolic groups have discrete spaces of directions.

\section{Basic concepts}\label{sec:basics}

\begin{definition}[hyperbolic group {[topological version]}]
\label{def:hyp-topG}
A topological group 
is called \emph{hyperbolic} 
if and only if 
it 
is compactly generated 
and 
its Cayley graph 
with respect to some 
(hence any)  
compact generating set 
is Gromov-hyperbolic. 
\end{definition} 

The definition makes sense,  
because 
all Cayley graphs 
with respect to 
compact generating sets 
are quasi-isometric 
by 
part~(\romannumeral1) of 
Lemma~4.6 
in~\cite{Gphs+Perm+topGs}. 
The same definition 
is used 
in recent work 
by 
Yves Cornulier 
and 
Romain Tessera
\cite{contrAut+L^p-Cohom<deg1}, 
where 
they characterize 
certain classes 
of non-discrete 
Gromov-hyperbolic groups.  

%
In this paper 
we consider compactly generated,
totally disconnected, locally compact topological groups 
only. 
These groups 
admit 
a locally finite, connected graph 
with a vertex-transitive action 
by the group 
such that 
vertex-stabilizers 
are 
compact and open. 
Such a graph 
with 
an action 
by the group 
is an instance of 
a so-called 
{rough Cayley graph},  
a concept 
introduced 
in~\cite{analog(CayleyGphs(topGs))}. 
We now 
define this concept. 

\begin{definition}\label{def:rough-Gayley-Gph}
Let~$G$ be 
a topological group. 
A connected graph~$X$ 
is said to be 
\emph{a rough Cayley graph} of~$G$,  
if 
$G$ acts transitively 
on the vertex set 
of~$X$ 
and 
the stabilizers of vertices 
are compact open subgroups 
of~$G$. 
\end{definition}

The proof 
of our main result relies on the existence of a rough Cayley graph for the groups under consideration. The relevant result is 
Theorem~2.2+ in~\cite{analog(CayleyGphs(topGs))}, or Corollary~1 in~\cite{FC-ele<tdGs&auts(infGphs)}, 
which we restate for ease of reference. 
In the formulation, 
$VX$ 
denotes 
the vertex set 
of~$X$. 

\begin{theorem}[Existence of a locally finite, rough Cayley graph]
\label{thm:existence(lf-roughCayleyGph)}
Let~$G$ be 
a totally disconnected, compactly generated, locally compact group. 
Then 
there is 
a locally finite, connected graph~$X$ 
such that: 
\begin{itemize}
\item[(\romannumeral1)] 
$G$ acts as 
a group of automorphisms 
on~$X$ 
and 
is transitive 
on~$VX$;
\item[(\romannumeral2)]
for every vertex~$v$ in~$X$ 
the subgroup~$G_v$ 
is 
compact and open 
in~$G$;
\item[(\romannumeral3)]
if~$\Aut X_.$ 
is equipped with 
the permutation topology,  
then 
the homomorphism 
$\pi\colon G\to \Aut X_.$ 
given by 
the action 
of~$G$ 
on~$X$ 
is continuous, 
the kernel of 
this homomorphism
is compact 
and 
the image 
of~$\pi$ 
is closed 
in~$\Aut X_.$.
\end{itemize}
Conversely, 
if~$G$ 
acts as 
a group of automorphisms on 
a locally finite, connected graph~$X$ 
such that~$G$ is 
transitive on 
the vertex set 
of~$X$ 
and 
the stabilizers of 
the vertices 
in~$X$ 
are
compact and open, 
then~$G$ is 
compactly generated.
\end{theorem}

For 
a totally disconnected, locally compact group 
hyperbolicity 
can be 
formulated 
in terms of 
any of 
its 
rough Cayley graphs 
as follows. 

\begin{proposition}[hyperbolicity in terms of the rough Cayley graph] 
\label{prop:hyperbolicity(roughCayley-graph)}
A totally disconnected, locally compact group 
is hyperbolic 
if and only if 
some 
(hence any) 
of its rough Cayley graphs 
is hyperbolic. 
\end{proposition} 
\begin{proof} 
The claim 
is implied by 
the quasi-isometry 
of rough Cayley graphs; 
see Theorem~4.5 
in~\cite{Gphs+Perm+topGs} 
or 
Theorem~2.7 
in~\cite{analog(CayleyGphs(topGs))}. 
\qed 
\end{proof}

The \flatrk\! 
of a group $\mathcal{A}$ of automorphisms of a 
totally disconnected locally compact group $G$ 
was introduced in \cite{tidy<:commAut(tdlcG)}, although it was not given that name there. 
Some auxiliary definitions and results will be required for its definition and in later sections. 
\begin{definition}
\label{defn:minimizing;flat}
Let $G$ be a totally disconnected, locally compact group. 
\begin{itemize}
\item[(\romannumeral1)]  The \emph{scale} of the automorphism, $\alpha$, of $G$ is the positive integer
\begin{equation}
\label{eq:scaledefn}
s_G(\alpha) := 
\min\left\{ |\alpha(O)\colon \alpha(O)\cap O| \colon O\leqslant G\text{ compact and open}\right\}.
\end{equation}
\item[(\romannumeral2)]  The compact, open subgroup $O$ is \emph{minimizing} for $\alpha$ if the minimum index in {\rm(\ref{eq:scaledefn})} is attained at  $O$.
\item[(\romannumeral3)] The group, ${\mathcal A}$ of automorphisms of $G$ is \emph{flat} if there is a compact open subgroup $O\leqslant G$ that is minimizing for every $\alpha\in {\mathcal A}$.
\end{itemize}
\end{definition}

\begin{theorem}[\cite{tidy<:commAut(tdlcG)}, Corollary 6.15]
\label{thm:flat}
Let ${\mathcal A}$ be a flat group of automorphisms of $G$ and $O$ be minimizing for ${\mathcal A}$. Then ${\mathcal A}_1 := \left\{\alpha\in{\mathcal A} \colon \alpha(O) = O\right\}$ is a normal subgroup of ${\mathcal A}$, and ${\mathcal A}/{\mathcal A}_1$ is a free abelian group. 
\end{theorem}
The group ${\mathcal A}_1$ is independent of the minimizing subgroup used to define it.
\begin{definition}
\label{defn:flat-rank}
Let $G$ be a totally disconnected, locally compact group.
\begin{itemize}
\item[(\romannumeral1)] The \emph{rank} of the flat group, ${\mathcal A}$ of automorphisms of $G$ is the rank of the free abelian group ${\mathcal A}/{\mathcal A}_1$. 
\item[(\romannumeral2)] The \emph{\flatrk} of a group ${\mathcal A}$ of automorphisms of $G$ is the supremum of the  ranks of all the flat subgroups of  ${\mathcal A}$. 
\item[({\romannumeral3})] The \emph{\flatrk} of $G$ is the \flatrk of the group of inner automorphisms. 
\end{itemize}
\end{definition}

\section{Constructing hyperbolic topological groups}\label{sec:supplement}

The following proposition 
provides a method 
to construct 
totally disconnected, locally compact, hyperbolic groups. 
For example 
one might take 
for~$X$ 
the Cayley graph 
of any 
discrete hyperbolic group 
and 
let~$G$ be 
its full automorphism group 
with the permutation topology 
(equivalently, 
the compact-open topology). 
Further applications 
of this result 
will be provided 
in Section~\ref{sec:examples}. 

\begin{proposition}
\label{prop:hyperbolicity(cocp-<Gs(Aut(lf-Gromov-hyp-complexes))}
Let $G$ be 
a totally disconnected, locally compact group 
acting cocompactly 
and 
with compact, open point stabilizers on 
a locally finite, connected Gromov-hyperbolic complex, 
$X$ say. 
Then 
the $1$-skeleton 
of~$X$ 
is 
quasi-isometric to 
a locally finite, connected, rough Cayley graph 
for~$G$ 
which is also Gromov-hyperbolic, 
and 
$G$ is a hyperbolic group.  
\end{proposition}
\begin{proof}
Since 
the group~$G$ 
acts 
cocompactly 
on~$X$, 
there is 
a finite subcomplex, 
$F$ say,  
of~$X$ 
whose $G$-translates 
cover~$X$. 
The group~$G$ 
is then 
generated by 
its subset 
$\{x\in G\colon x.F\cap F\neq\emptyset\}$; 
this subset 
is compact, 
hence~$G$ 
is 
compactly generated. 

Denote 
the $1$-skeleton 
of~$X$ 
by~$\Gamma$. 
The graph~$\Gamma$ 
is 
a locally finite, connected $G$-set 
with compact, open point stabilizers. We use a standard argument, see \cite[Theorem~4.9]{Gphs+Perm+topGs} for instance, to complete the proof. 
Since 
the $G$-translates of the 1-skeleton of the finite complex, $F$, of the previous paragraph cover $\Gamma$, there are finitely many $G$-orbits in $\Gamma$ and there is a constant $k$ (for example, the diameter of $F$)  such that every vertex is within distance $k$ of each orbit. Fix a vertex $x$ of $\Gamma$ and define a graph structure on the orbit $G.x$ by drawing an edge between vertices $g.x$ and $h.x$ if they are within distance $2k+1$. Then the graph $G.x$ is connected, locally finite and quasi-isometric to $\Gamma$. Hence $G.x$ is Gromov-hyperbolic. Furthermore, $G$ acts transitively with compact vertex stabilizers on this graph. Therefore the graph $G.x$ is a rough Cayley graph for $G$ and $G$ is a hyperbolic group. 
\qed 
\end{proof}

While 
the action 
of a group 
on its Cayley graph 
is always faithful, 
the action of 
a compactly generated topological group 
on its rough Cayley graph 
need not be. 
The following result 
explains 
what happens 
when 
passing to 
the quotient 
of the group 
by the kernel of 
this action. 
The proof 
is straightforward,  
and 
is 
left to 
the reader. 

\begin{proposition}\label{prop:quot-by-kernel-of-action-on-rough-Cayley-graph}
Let~$G$ be 
a compactly generated topological group 
that contains 
a compact, open subgroup 
and 
let~$\Gamma$ be  
any of 
its 
locally finite, connected, vertex-transitive 
rough Cayley graphs 
with 
compact, open vertex-stabilizers. 
Let~$\widehat{G}$ be 
the quotient 
of~$G$ 
by 
the compact kernel of 
its action 
on~$\Gamma$. 
Then~$\Gamma$ 
with its induced action 
by~$\widehat{G}$ 
is again 
a rough Cayley graph 
for~$\widehat{G}$ 
with 
the same properties. 
In particular, 
if 
the group~$G$ 
is hyperbolic, 
then 
so is 
the group~$\widehat{G}$.  
\qed 
\end{proposition}

We also have 
the following 
transfer result 
for flat subgroups 
under continuous, open, surjective homomorphism 
with compact kernel. 

\begin{proposition}\label{prop:flat<Gs&proper-maps}
Let~$\pi\colon G\to \widehat{G}$ be 
a continuous, open, surjective homomorphism 
with compact kernel 
between 
totally disconnected, locally compact groups.
Further, 
let~$H$ be 
a flat subgroup 
of~$G$  
and~$\widehat{H}$ 
a flat subgroup 
of~$\widehat{G}$. 
Then~$\pi(H)$ 
is a flat subgroup 
of~$\widehat{G}$ 
of the same rank 
as~$H$ 
while 
$\pi^{-1}(\widehat{H})$ 
is  
a flat subgroup 
of~$G$ 
of the same rank 
as~$\widehat{H}$. 
The groups~$G$ 
and~$\widehat{G}$ 
have the same flat rank.  
\end{proposition}
\begin{proof}
Consider $h\in G$ and suppose that $O$ is minimizing for $h$, that is, is minimizing for the inner automorphism $\alpha_h : x\mapsto hxh^{-1}$. The index $|hOh^{-1}\colon hOh^{-1}\cap O|$ is unchanged if $O$ is replaced by $O\ker(\pi)$ and so it may be assumed 
that 
$\ker(\pi)\subseteq O\cap hOh^{-1}$.  
%
The subgroup~$\pi(O)$ 
of~$\widehat{G}$,  
which is 
also compact and open, 
then satisfies  
\begin{equation}\label{eq:tidyness_under-proper-maps}
|hOh^{-1}\colon hOh^{-1}\cap O|=
|\pi(h)\pi(O)\pi(h)^{-1}\colon \pi(h)\pi(O)\pi(h)^{-1}\cap \pi(O)|\,, 
\end{equation}
from which it follows that~$s_{\widehat G}(\pi(h))\leq s_G(h)$. On the other hand, if ${\widehat O} \leqslant {\widehat G}$ is minimizing for ${\hat h}\in {\widehat G}$, then $\pi^{-1}({\widehat O})$ is compact and open in $G$ and 
\begin{equation}\label{eq:tidyness_under-proper-maps2}
|h\pi^{-1}({\widehat O})h^{-1}\colon h\pi^{-1}({\widehat O})h^{-1}\cap \pi^{-1}({\widehat O})|=
|{\hat h}{\widehat O}{\hat h}^{-1}\colon {\hat h}{\widehat O}{\hat h}^{-1}\cap {\widehat O}|\,, 
\end{equation}
for any $h\in G$ with $\pi(h) = {\hat h}$ and it follows that~$s_G(h) \leq  s_{\widehat G}(\pi(h))$. Therefore the scales are equal and $O$ is minimizing for $h$ if and only if $\pi(O)$ is minimizing for $\pi(h)$. 

Letting $H$ be a flat subgroup of $G$ and ${\widehat H}$ be a flat subgroup of ${\widehat G}$, it follows that $\pi(H)$ and $\pi^{-1}({\widehat H})$ are flat subgroups of $G$ and ${\widehat G}$ respectively as claimed. Moreover, $h\in H_1$ if and only if $\pi(h_1)\in \pi(H)_1$ and so the map $\pi : H\to \pi(H)$ induces an isomorphism $H/H_1\to \pi(H)/\pi(H)_1$. Hence the ranks of $H$ and $\pi(H)$ are equal as claimed. That the ranks of ${\widehat H}$ and $\pi^{-1}({\widehat H})$ also agree may be seen similarly. Therefore $G$ and ${\widehat G}$ have the same \flatrk\!.
\qed
\end{proof}


\section{Method of proof}\label{sec:method(proof)}

The proof of 
the main result 
uses 
the classification 
of group actions by isometries
on Gromov-hyperbolic spaces in terms of fixed point properties 
versus 
existence of free subgroups.  
This is combined with 
topological properties 
of elliptic, parabolic and hyperbolic isometries 
and  
an analysis of 
the dynamics of 
actions of flat subgroups 
on the boundary 
of the hyperbolic space. 
For the geometric ideas we follow the approach 
in~\cite{woess}.

%
We begin by 
extending 
the necessary concepts 
of hyperbolic geometry 
to encompass 
topological groups. 

\begin{definition}[boundary of a hyperbolic group]  
Let~$G$ be 
a hyperbolic topological group. 
The \emph{hyperbolic boundary  
of~$G$} 
is 
the Gromov-boundary of 
the Cayley graph 
of~$G$ 
with respect to 
some compact generating set 
of~$G$. 
\end{definition}

The usual properties 
of the hyperbolic boundary 
carry over from 
the discrete case. 
That it is independent of the rough Cayley graph chosen
will be important 
in subsequent 
arguments. 

\begin{proposition}\label{prop:prop(hyp-boundary(G))}
The hyperbolic boundary of 
a hyperbolic topological group, $G$,
is 
independent of  
the choice of compact generating set 
used for 
its definition. 
It  
is 
a metric space 
which 
admits 
an action 
of $G$ 
by bi-Lipschitz maps. 
If 
$G$ 
admits 
a compact, open subgroup, 
then 
metrics 
can be chosen 
such that 
its hyperbolic boundary 
is equivariantly isometric to 
the Gromov-boundary 
of any 
of its rough Cayley graphs;  
in particular, 
in that case 
the hyperbolic boundary 
of $G$ 
is compact.  
\end{proposition} 
\begin{proof}
Using 
standard results 
about Gromov-hyperbolic spaces, 
the above statements  
follow from 
Theorem~4.5 
and 
Lemma~4.6 
in~\cite{Gphs+Perm+topGs}. 
\qed 
\end{proof}

The classification of 
isometries of 
hyperbolic spaces 
also plays 
a central role 
in what follows. 
Since 
the action 
of a topological group 
on a 
rough Cayley graph 
need not be faithful, 
we  
extend 
the usual definitions 
as follows. 

\begin{definition}[elliptic, parabolic and hyperbolic elements]
\label{def:elliptic,parabolic,hyperbolic}
Let~$G$ be 
a group, 
$X$ be 
a Gromov-hyperbolic space  
and 
$\alpha\colon G\to \Aut X_.$
be 
an action 
of~$G$ 
on~$X$ 
by isometries. 
An element~$g$ 
in~$G$ 
is called 
\begin{enumerate}
\item 
\emph{$\alpha$-elliptic} 
if
there is 
a point 
of~$X$ 
whose $\alpha(g)$-orbit 
is bounded; 
in that case 
every other point 
of~$X$ 
has 
the same property; 
\item 
\emph{$\alpha$-parabolic} 
if 
it is not $\alpha$-elliptic 
and~$\alpha(g)$ fixes 
a unique boundary point;  
%
\item 
\emph{$\alpha$-hyperbolic} 
if 
it is not $\alpha$-elliptic 
and~$\alpha(g)$ fixes 
precisely two boundary points, 
which, 
for 
arbitrary~$x\in X$
are then 
of the form 
$\lim_{n\to \infty} \alpha^n(g).x$, 
called \emph{attracting for~$g$},  
and 
$\lim_{n\to \infty} \alpha^{-n}(g).x$, 
called \emph{repelling for~$g$}. 
\end{enumerate}
An element 
of a hyperbolic topological group 
is called 
\emph{elliptic, parabolic or hyperbolic} 
respectively, 
if 
it is 
$\alpha$-elliptic, $\alpha$-parabolic or $\alpha$-hyperbolic 
respectively 
for 
$\alpha$ 
equal to 
the natural action 
of the group 
on 
its Cayley graph 
with respect to 
a compact generating set.  
\end{definition}

Since 
Cayley graphs 
with respect to 
compact generating sets 
are 
quasi-isometric, 
the notions 
elliptic, parabolic or hyperbolic 
do 
not depend on 
the particular Cayley graph 
chosen and 
reference to
the homomorphism~$\alpha$ 
is usually 
omitted in the following. 

The elliptic elements 
of a locally compact, hyperbolic group 
can be characterized 
by an intrinsic 
topological property 
which 
generalizes 
the corresponding characterization 
in the discrete case.  

\begin{proposition}\label{prop:top-char(elliptics)}
An element of 
a locally compact, hyperbolic, topological group 
is elliptic 
if and only if 
it is topologically periodic.
\end{proposition}
\begin{proof}
By definition, 
an element, 
$g$ say, 
of the given group~$G$ 
is elliptic 
if and only if 
its orbits 
in the Cayley graph, 
$\Gamma$ say,  
of~$G$ 
with respect to 
a compact set 
of generators 
is bounded 
in the graph metric. 
The property 
of being bounded 
is independent of 
the orbit chosen. 
Hence 
$g$ is elliptic 
if and only if 
its orbit $\langle g\rangle.e= \langle g\rangle$ 
is bounded. 
Abels' result~2.3 
in~\cite{Specker-cp(lctG)} 
(Heine-Borel-Eigenschaft) 
implies that 
a subset 
of~$\Gamma$ 
is bounded 
in the graph-metric 
if and only if 
it is 
a relatively compact subset 
of~$G$. 
The latter condition 
is satisfied by 
the set $\langle g\rangle$ 
if and only 
$g$ is topologically periodic. 
%
\qed
\end{proof}

\section{Properties of elliptic, parabolic and hyperbolic elements}\label{sec:types of elements}

\subsection{The scale of an elliptic element is~$1$}\label{subsec:scale(elliptic)=1}

An elliptic element in a totally disconnected, locally compact group is topologically periodic, by Proposition~\ref{prop:top-char(elliptics)}. Hence the set of conjugates of an open subgroup by powers of such an element is finite and the intersection of these conjugates is an open subgroup normalized by the element.

\begin{proposition}\label{prop:elliptic->scale1}
The scale of 
every elliptic element 
in 
a totally disconnected, locally compact, hyperbolic group 
is~$1$. 
\qed
\end{proposition}

\subsection{Totally disconnected, 
hyperbolic groups contain no parabolics}\label{subsec:hypGs-not>parabolics}

Discrete hyperbolic groups 
do not contain 
parabolic elements. 
This is proved in 
each one 
of the following sources: 
\cite[Corollary~8.1.D]{hypGs} (together with the obvious observation that torsion elements are elliptic), \cite[Chapitre~9, Th\'eor\`eme~3.4]{LNM1441} 
and 
\cite[Chapitre~8, Th\'eor\`eme~29]{surGhyp-d'apresGromov}. 

Theorem~\ref{thm:Aut(hyperbolic-complex)-has-no-parabolics} below 
extends 
this result 
to totally disconnected, locally compact, hyperbolic groups, 
the proof of which is modelled on 
the argument from~\cite{LNM1441}. 
The following property 
of non-hyperbolic elements 
is of central importance 
in the proof. 

\begin{lemma}\label{lem:parabolics-bounded-at-infty}
Let $X$ be 
a geodesic $\delta$-hyperbolic space. 
Then 
there is 
a constant~$k$, 
depending only on~$\delta$,  
such that: 
given a non-hyperbolic isometry, $g$, 
of~$X$ 
that fixes 
a boundary point~$\omega$, 
and a geodesic ray, $x$, 
ending in~$\omega$,  
all points 
on $x$ 
which are 
sufficiently far out 
are moved 
by a distance of 
at most~$k$ 
by~$g$. 
\end{lemma}
\begin{proof} 
Choose any point~$p$  
of~$X$ 
and 
denote 
the midpoint 
of a chosen geodesic segment 
connecting 
$p$ to $g.p$ 
by~$m$.   
By Lemme~9.3.1 
in~\cite{LNM1441} 
\[
d(g.m,m)\le 6\delta\,.
\]
Applying 
Lemme~9.3.6 
in~\cite{LNM1441}
to the entities $x$ and $m$ 
chosen,  
we obtain that 
there is 
a number~$t_0\ge 0$ 
such that 
for each~$t\ge t_0$ 
\[
d(g.x(t),x(t))\le 72\delta +d(g.m,m)\le 72\delta+6\delta=78\delta\,,
\]
and so we may take $k$ to be $78\delta$. 
\qed
\end{proof}

The following theorem 
is 
the main result 
of this subsection.

\begin{theorem}\label{thm:Aut(hyperbolic-complex)-has-no-parabolics}
Suppose that 
a group~$G$  
acts cocompactly 
and 
by automorphisms 
on a connected, locally finite, metric, Gromov-hyperbolic complex. 
Then~$G$  
does not contain 
parabolic elements. 
In particular, 
a hyperbolic topological group 
with a compact, open subgroup 
does not contain 
parabolic elements. 
\end{theorem}
\begin{proof}
We argue 
by contradiction,  
and assume that 
some element, 
$g$ say, 
of~$G$ 
acts by 
a parabolic isometry. 
All positive powers of $g$ 
are again 
parabolic 
and have the same 
unique fixed point 
on the boundary, 
$\omega$ say.  

The diameter 
of cells 
in our space 
is bounded 
from above, 
say by
the positive number~$D$, 
because 
the group~$G$ 
is assumed 
to act 
cocompactly. 
Denote by~$n$ 
a natural number 
larger than 
the maximal number 
of vertices 
of the space 
that are 
contained in 
any closed ball 
whose radius is~$k+2D$, 
where~$k$ is  
the number
introduced in 
the statement 
of Lemma~\ref{lem:parabolics-bounded-at-infty}. 
Such a number~$n$ 
exists, 
because~$G$ acts cocompactly. 

Choose 
a geodesic ray, 
$x$ say, 
that ends 
in~$\omega$. 
Lemma~\ref{lem:parabolics-bounded-at-infty} 
applied to 
the elements $g, g^2,\ldots, g^n$ 
and this~$x$ 
implies that 
there is a 
number~$T$ 
such that 
for 
$t\ge T$
and~$i=1,\ldots, n$ 
we have 
$ 
d(g^i.x(t),x(t))\le k 
$.  
By the definition 
of~$D$,  
there is 
a vertex, 
$v$ say,  
at distance 
at most~$D$ 
from $x(t)$. 
Then 
the above bound 
on the displacement of 
the point~$x(t)$,  
implies that~$v$ 
is moved 
a distance 
at most $k+2D$ 
by each of 
the elements $g$, $g^2$, \ldots $g^n$. 

By our choice 
of~$n$,  
there are exponents 
$i<j$ 
such that 
$g^i.v=g^j.v$. 
But then 
$g^{j-i}$ fixes~$v$,  
and hence 
is elliptic. 
Hence 
the element~$g$ 
is also elliptic, 
in contradiction 
to the assumption 
that 
$g$ is parabolic. 
This contradiction 
shows that 
there is no parabolic element, 
finishing the proof. 
\qed 
\end{proof}

\section{Proof of the Main Result}\label{sec:Pf(MainResult)}

We will use 
the classification 
of actions 
on hyperbolic spaces 
established in~\cite{woess},  
as already mentioned.   
%
The bound on
the rank of 
a flat subgroup is proved
on a case-by-case basis according to
this classification. 

The following two lemmas  
prepare 
Proposition~\ref{prop:F2<flatG=>rk=0}, 
which 
provides 
this bound 
if the flat group 
contains 
a non-abelian free group 
consisting of 
hyperbolic elements.

\begin{lemma}\label{lem:Gv notfix rep.pt(elt.scale>1)}
Let $G$ be 
a totally disconnected, locally compact group 
acting cocompactly 
and 
with compact, open point stabilizers on 
a locally finite, connected $\delta$-hyperbolic complex. 
If~$h$ is 
an element 
of~$G$ 
of non-trivial scale 
(which is thus necessarily hyperbolic),  
then 
the orbit of~$\omega_h$, 
the repelling boundary point  
of~$h$, 
under 
every open subgroup 
is infinite. In particular, 
no open subgroup of $G$
 fixes~$\omega_h$. 
\end{lemma}
\begin{proof}
It follows from Proposition~\ref{prop:elliptic->scale1} 
and 
Theorem~\ref{thm:Aut(hyperbolic-complex)-has-no-parabolics} 
that 
an element~$h$ 
in~$G$ 
of non-trivial scale 
must 
indeed 
be hyperbolic 
as stated. 
Proposition~\ref{prop:hyperbolicity(cocp-<Gs(Aut(lf-Gromov-hyp-complexes))} 
implies that~$h$ 
also 
acts as 
a hyperbolic automorphism 
of the given complex. 

We now begin 
the proof proper; 
we will prove 
the contraposition. 
Assume then 
that 
$V$ is 
an open subgroup 
of~$G$ 
such that 
the orbit 
of 
$\omega_h$ 
under~$V$ 
is finite. 
Then 
a closed subgroup 
of finite index 
in~$V$ 
fixes~$\omega_h$ 
and 
we can assume that 
$V$ fixes~$\omega_h$. 
Intersecting
the group~$V$ 
with 
the stabilizer 
of a vertex, 
$v$ say, 
we may assume 
that 
$V$ fixes 
a given vertex~$v$ also. 

Applying Theorem~7.7 in~\cite{struc(tdlcG-graphs+permutations)} 
with~$V$ 
equal to 
the group of the same name, 
and~$x$ equal to~$h$, 
we see that 
the scale 
of~$h$ 
is given by 
the limit 
\[
\lim_{n\to\infty} |V.(h^{-n}.v)|^{1/n}\,. 
\]
We will use 
our assumptions 
to show that 
there is 
a bound 
on the diameter 
of the orbits~$V.(h^{-n}.v)$  
that is 
uniform in~$n\in\mathbb{N}$. 
Because 
$G$ acts 
cocompactly, 
this implies that 
there is 
a uniform bound on 
the number 
of the vertices  
in these orbits. 
The displayed formula above 
will then show 
that the scale 
of~$h$ 
is~$1$,  
and 
establish 
our claim.  

The map~$f$ 
that sends
an integer~$n$ 
to the vertex~$h^{-n}.v$ 
is 
a quasi-geodesic ray 
that converges 
to~$\omega_h$. 
By 
part~$(i)$ of Th\`eor\'eme~5.25 in~\cite{surGhyp-d'apresGromov}, 
there is a geodesic ray, $r$ say,  
that starts at~$v$   
and is at Hausdorff-distance 
at most~$H$ 
from~$f$;  
the ray~$r$ 
therefore 
ends in~$\omega_h$ 
also. 
Then part~$(i)$ of Corollaire~7.3 in~\cite{surGhyp-d'apresGromov}  
implies that 
$d(r(t),g.r(t))\le 8\delta$ for all $g\in V$ and all $t\ge 0$.

Let~$n$ be 
an integer. 
On the geodesic ray~$r$ 
choose 
a point, 
$r(t_n)$ say, 
such that 
the distance 
between~$h^{-n}.v$ and~$r(t_n)$ 
is 
at most~$H$. 
For every~$g$ 
in~$V$ 
the distance 
between~$g.(h^{-n}.v)$ and~$g.r(t_n)$ 
is then 
at most~$H$ 
also.   
%
%
We conclude that 
for all~$g$ 
in~$V$ 
the distance 
between 
$h^{-n}.v$ and~$g.(h^{-n}.v)$ 
is at most~$2H+8\delta$; 
hence 
the diameter of 
the orbits~$V.(h^{-n}.v)$  
is 
indeed 
uniformly bounded, 
and 
we are done. 
\qed 
\end{proof}

The claim of 
the following Lemma 
is false 
for flat groups 
of rank~$0$. 

\begin{lemma}\label{lem:flatG(rk>0)=>no-hyperb<unisc<G}
Let $G$ be 
a totally disconnected, locally compact, hyperbolic group 
and 
$H$ a flat subgroup 
of~$G$ 
of rank 
at least~$1$. 
Then 
the group~$H_1$ 
does not contain 
hyperbolic elements. 
\end{lemma}
\begin{proof} 
We will derive 
a contradiction 
to Lemma~\ref{lem:Gv notfix rep.pt(elt.scale>1)} 
from 
the assumption 
that 
there is 
a hyperbolic element
in~$H_1$  
and 
an (automatically hyperbolic) element
in $H\smallsetminus H_1$, 
thus 
establishing the claim.

Let $O$ be 
a minimizing subgroup 
for~$H$. 
The group~$O$ 
is the stabilizer 
in~$G$ 
of the point~$O$ 
in any rough Cayley graph 
of~$G$ 
constructed from 
$O$. 
The group $H_1$ 
normalizes
the subgroup~$O$. 
Hence 
the group~$O$ 
fixes 
the whole orbit 
$H_1.O$ 
pointwise. 
All boundary points 
fixed by hyperbolic elements 
in~$H_1$ 
are limit points 
of this orbit, 
hence 
the set 
of these points, 
$L_{hyp}(H_1)$ say, 
is fixed by 
the group~$O$ 
also. 

Choose 
a hyperbolic element~$k$ 
in $H_1$. 
Denote 
the repelling boundary point 
of~$k$ 
by $\omega_k\in L_{hyp}(H_1)$.  
Further 
choose  
an element~$h$ 
in $H\smallsetminus H_1$; 
which 
is possible 
because 
the rank 
of~$H$ 
is 
at least~$1$. 
Replacing, 
if necessary,  
$h$ by its inverse, 
we may assume that 
$h$ has 
non-trivial scale. 
Proposition~\ref{prop:elliptic->scale1} 
and 
Theorem~\ref{thm:Aut(hyperbolic-complex)-has-no-parabolics} 
imply that 
the element~$h$ 
is hyperbolic.  
%
Denote 
the repelling boundary point 
of~$h$ 
by~$\omega_h$. 

Since 
the subgroup~$O$ 
fixes 
$\omega_k\in L_{hyp}(H_1)$ 
and 
$h$ has 
non-trivial scale,  
$\omega_k$ is different from $\omega_h$, 
for otherwise 
we would have 
a contradiction 
to Lemma~\ref{lem:Gv notfix rep.pt(elt.scale>1)}.  
Because 
$H_1$ is 
normal 
in~$H$, 
the group~$H$ 
leaves 
the set $L_{hyp}(H_1)$
invariant. 
In particular, 
the sequence $\bigl(h^{-n}(\omega_k)\bigr)_{n\in\mathbb{N}}$, 
which converges to $\omega_h$ 
since $\omega_k\neq \omega_h$, 
is contained 
in~$L_{hyp}(H_1)$. 
By continuity 
of the action of $G$ 
on the hyperbolic compactification,  
$\omega_h$ 
is contained 
in~$L_{hyp}(H_1)$  
and hence 
is fixed by~$O$. 
This is 
the anticipated contradiction 
to Lemma~\ref{lem:Gv notfix rep.pt(elt.scale>1)}. 
\qed
\end{proof} 

We are 
now 
ready 
to treat 
the first case 
in the classification.  

\begin{proposition}\label{prop:F2<flatG=>rk=0}
Let $G$ be 
a totally disconnected, locally compact, hyperbolic group 
and 
$H$ a flat subgroup 
of $G$ 
that contains 
a non-abelian free group 
consisting of hyperbolic elements. 
Then the rank of~$H$ 
is~$0$.
\end{proposition} 
\begin{proof}
Let $F$ be 
a non-abelian free subgroup 
of~$H$ 
consisting of 
hyperbolic elements. 
Assume 
by way of contradiction 
that 
the rank 
of~$H$ 
is at least~$1$. 
Using Lemma~\ref{lem:flatG(rk>0)=>no-hyperb<unisc<G}, 
we 
then 
conclude that 
the subgroup $H_1$ 
contains no hyperbolic elements. 
Then 
the restriction of 
the canonical map 
$H\to H/H_1$ 
to 
the subgroup~$F$ 
of~$H$ 
has trivial kernel. 
It follows that the abelian group~$H/H_1$ 
contains 
a non-abelian free group, 
which is absurd. 
Therefore the rank 
of~$H$ 
is~$0$ 
as claimed. 
\qed 
\end{proof}

The second case 
of the classification 
is easy. 

\begin{proposition}\label{prop:flatG-stab-nonempty-cpS=>rk=0} 
Let $G$ be 
a totally disconnected, locally compact, hyperbolic group 
and 
$H$ a flat subgroup 
of~$G$ 
that 
stabilizes 
a non-empty, compact subset 
of some rough Cayley graph 
of~$G$. 
Then 
the rank 
of~$H$ 
is~$0$. 
\end{proposition} 
\begin{proof} 
The condition 
satisfied by~$H$ 
implies that 
all elements 
of~$H$ 
are elliptic. 
By  Proposition~\ref{prop:elliptic->scale1}, 
$H$ is 
contained in 
its subgroup~$H_1$ 
and 
the \flatrk\! 
of~$H$ 
is~$0$ 
as claimed. 
\qed 
\end{proof}

The next lemma 
proves that 
a quasi-geodesic ray 
converging to 
a boundary point, 
$\omega$ say, 
is uniformly close to 
any geodesic ray 
converging to~$\omega$. 
This lemma 
is used in  
the last cases 
in the classification, 
Proposition~\ref{prop:flatG-fix-1OR2boundary-pts=>rk le1} 
below. 

\begin{lemma}\label{lem:q-geodesic rays hit close near their boundary point}
Given 
real numbers~$\delta\ge 0$, $\lambda\ge 1$ and~$c\ge0$ 
there is 
a  constant~$R$ 
(depending on~$\delta$, $\lambda$ and~$c$ only) 
such that 
for 
any proper $\delta$-hyperbolic space, $X$: 
given a geodesic ray, $r$,  
and a $(\lambda, c)$-quasi-geodesic ray, $f$, in $X$
that converge to the same boundary point~$\omega$,  
the image of $f$ intersects 
the ball 
of radius~$R$ 
centred on any point 
sufficiently far out 
on $r$. 
\end{lemma}
\begin{proof}
By 
part~$(i)$ of Th\`eor\'eme~5.25 in~\cite{surGhyp-d'apresGromov} 
there is a geodesic ray~$g$ 
at Hausdorff distance 
at most~$H$ 
from~$f$, 
where~$H$ 
depends 
on~$\delta$, $\lambda$ and~$c$ only. 
The geodesic ray~$g$ 
also 
converges to 
the boundary point~$\omega$. 
Hence, according to Proposition~7.2 in~\cite{surGhyp-d'apresGromov} 
appropriate subrays 
$r'$ and $g'$ 
of the respective rays~$r$ and~$g$ 
have 
Hausdorff distance 
at most~$16\delta$. 
Then for each point on $r'$, the ball with radius~$16\delta$ centred on that point intersects $g'$, and for each point on $g'$ the ball with radius~$H$ centred on that point intersects $f$. Hence the claim holds with $R = H + 16\delta$. 
\qed
\end{proof}

Finally, 
we cover  
the last two cases 
of the classification.

\begin{proposition}\label{prop:flatG-fix-1OR2boundary-pts=>rk le1}
Let $G$ be 
a totally disconnected, locally compact, hyperbolic group 
and 
$H$ a flat subgroup 
of~$G$ 
that fixes 
a boundary point 
or 
a pair of boundary points 
(not necessarily pointwise). 
Then 
the rank 
of~$H$ 
is at most~$1$. 
\end{proposition} 
\begin{proof}
We first 
reduce 
to the case 
where 
the flat subgroup~$H$ 
fixes 
a boundary point. 
The other case is where~$H$ 
fixes a pair of distinct boundary points 
without fixing the points. 
Then 
the subgroup 
of~$H$ 
that fixes 
both points 
is also flat 
and 
has index~$2$ 
in~$H$. 
Hence this subgroup
has the same rank 
as~$H$ and it suffices 
to prove the claim 
for  
it. 

Next 
we show that 
the images 
of any two 
hyperbolic elements, 
$g$ and $h$ say, 
that both fix 
a boundary point 
satisfy 
a nontrivial relation 
in~$H/H_1$, 
thus 
showing that 
$H$ can not contain 
two elements 
mapping to 
linearly independent 
elements 
in the quotient 
and 
thereby 
finishing 
the proof. 

Inverting 
one of $g$, $h$ 
if necessary, 
we may assume that 
$g$ and $h$ have 
the same attracting boundary point, 
$\omega$ say.  
Choose 
a vertex, 
$v$, 
in a rough Cayley graph, 
$\Gamma$,  
for~$G$. 
The map $f_h\colon \mathbb{N}\to \Gamma$ 
defined by $f_h(n):=h^n.v$ 
is a quasi-geodesic ray,  
with quasi-isometry-constants $(\lambda, c)$ 
say. 
Then all the maps $g^i.f_h$ with $i\geq0$
are $(\lambda, c)$-quasi-geodesic rays 
and
converge to~$\omega$. 

Choose a geodesic ray, 
$r$ say,  
ending in~$\omega$.
Denote by~$n$ 
a natural number 
larger than 
the maximal number 
of vertices 
of~$\Gamma$
that are 
contained in 
any closed ball 
whose radius is
the constant~$R$ 
provided by Lemma~\ref{lem:q-geodesic rays hit close near their boundary point}. 
Such a number~$n$ 
exists, 
because 
$G$ acts cocompactly 
on its rough Cayley graph~$\Gamma$. 
According to Lemma~\ref{lem:q-geodesic rays hit close near their boundary point}, 
we may choose 
a point 
sufficiently far out 
on the ray~$r$ 
such that 
all the quasi--geodesic rays $g.f_h,\ldots, g^n.f_h$ 
intersect the ball~$B$  
of radius~$R$ 
around it. 
All points 
of intersection 
of $g.f_h,\ldots, g^n.f_h$ 
with~$B$ 
are vertices and so, by our choice 
of~$n$,  
there are 
integers~$i$ and~$j$ 
with $0<i<j$ 
such that 
$g^i(h^p.v)=g^j(h^q.v)$ 
for some integers~$p$ and~$q$. 

%
The element~$h^{-q}g^{i-j}h^p$ 
fixes~$v$, 
hence is elliptic 
and 
has scale~$1$. 
The relation 
$(p-q)\,hH_1+(i-j)\,gH_1=0$ 
therefore holds 
in~$H/H_1$ and, 
since $j-i\neq0$, 
$hH_1$ and $gH_1$ 
are linearly dependent. 
\qed 
\end{proof}

The main result 
of the paper 
is 
now 
obtained by combining these cases. 

\begin{theorem}
The \flatrk\! of 
a totally disconnected, locally compact, hyperbolic group 
is at most~$1$. 
\end{theorem} 
\begin{proof}
Let $G$ be 
a totally disconnected, locally compact, hyperbolic group 
and 
$H$ a flat subgroup 
of~$G$. 
Choose 
a connected, locally finite, rough Cayley graph,  
say~$\Gamma$, 
for~$G$. 
The graph~$\Gamma$ 
is Gromov-hyperbolic and \S4C of~\cite{woess} explains how the results of that paper apply to $\Gamma$ and its hyperbolic compactification. In particular, \cite[Theorem~3]{woess} 
lists 
the possible types 
of actions 
for~$H$ on $\Gamma$.  
Each  
of these possible types 
is covered by 
either Proposition~\ref{prop:F2<flatG=>rk=0} 
(type~$(a)$), 
Proposition~\ref{prop:flatG-stab-nonempty-cpS=>rk=0} 
(type~$(b)$) 
or 
Proposition~\ref{prop:flatG-fix-1OR2boundary-pts=>rk le1} 
(types~$(c)$ and~$(d)$) and the rank 
of~$H$ 
is seen to be 
at most~$1$ 
in all cases. 
\qed
\end{proof}

\section{Flat subgroups, space of directions and hyperbolic boundary}\label{sec:furtherProp(flatGs)}

In this section 
we present 
the proof of 
Theorem~\ref{thm:flatrk-via-directions}, which shows that discreteness of the space of directions (defined in \cite{direction(aut(tdlcG))}) also imposes a bound of~1 on the \flatrk\!. The proof is followed by two conjectures that propose further links between flat subgroups of a hyperbolic, totally disconnected, locally compact group, its space of directions and hyperbolic boundary.

\begin{proof}[of {Theorem~\ref{thm:flatrk-via-directions}}]
\label{pref:pf(thm:flatrk-via-directions)}
The space of directions 
of a totally disconnected, locally compact group 
of \flatrk\!~$k$ 
contains a $k$-cell  
by Proposition~23 
in~\cite{direction(aut(tdlcG))}.  
Hence 
a group 
with a discrete space of directions 
can have \flatrk\! 
at most~$1$. 
Furthermore, 
since 
a group 
has \flatrk\!~$0$ 
if and only if 
its space of directions 
is empty, 
we even have that 
a group with 
a non-empty, discrete space of directions 
has \flatrk\! 
equal to~$1$. 
\qed 
\end{proof}

The argument in the above proof in fact shows that the \flatrk\! of a group is bounded by the dimension of the space of directions. This pseudo-metric space need not be finite dimensional manifold however, as the example (a group having \flatrk\!~1) in~\cite[5.2.3]{direction(aut(tdlcG))} illustrates. 

Theorem~\ref{thm:flatrk-via-directions} applies to closed subgroup of the automorphism group of a locally finite tree, by Proposition~36(2) 
in~\cite{direction(aut(tdlcG))}. 
In fact, the bound on the flat-rank of such groups 
also follows from Theorem~\ref{MainThm-variant} 
in case 
they are 
compactly generated, 
because they are then hyperbolic.\\
\begin{proposition}
\label{prop:tree-->hyperbolic}
Let $G$ be a compactly generated topological group acting minimally on a
 tree, $X$, such that the stabilizers of vertices are open subgroups of  $G$.
 Then $G$ acts with finitely many orbits on the vertices.
\end{proposition}
\begin{proof}
The argument follows that of \cite[Proposition~7.9(b)]{Bass} which establishes the claim for discrete $G$.  Let $\mathbf{A}$ denote the graph of groups arising from the
 action of $G$ on $X$. Let $G_{\mathbf{B}}$ denote the fundamental group of a subgraph of groups $\mathbf{B}$ of $\mathbf{A}$.  Note
that $G_{\mathbf{B}}$ is always an open subgroup of~$G$.  Clearly the groups $G_{\mathbf{B}}$, where $G_{\mathbf{B}}$ ranges over all finite subgraphs of groups form an open covering of $G$. Because $G$ has a a compact generating set we see that finitely many of the groups $G_{\mathbf{B}}$, with $\mathbf{B}$ a finite subgraph of groups, cover the generating set. Indeed, one sees from this that there is a finite subgraph of groups, $\mathbf{A}'$,
 such that the fundemental group of $\mathbf{A}'$ contains the generating set and thus
 the fundamental group of $\mathbf{A}'$ is equal to $G$.  By \cite[Proposition~7.12]{Bass} we can now conclude that because the action is minimal that $\mathbf{A}' =\mathbf{A}$.  The graph of groups $\mathbf{A}'$ is finite and thus $\mathbf{A}$ is also finite and hence the group $G$ has only finitely many orbits on both the vertices and edges of $X$.
\qed
\end{proof}
\begin{corollary}
\label{cor:tree-->hyperbolic}
Let $G$ be a compactly generated group acting on a tree, $X$, such that the stabilizers of vertices
 are compact open subgroups of $G$. Then $G$ is hyperbolic.
\end{corollary}
\begin{proof}
If $G$ consists only of elliptic elements, then, since $G$ is compactly generated, \cite[Proposition~7.2]{Bass} implies that it is compact. Hence $G$ is in this case trivially hyperbolic.

Otherwise, $G$ contains a hyperbolic element and the union of all axes of all hyperbolic elements is a minimal $G$-invariant subtree of $X$. Replacing $X$ by this subtree and $G$ by its quotient by the (compact) stabilizer of this subtree, it may be assumed that the action is minimal.  By Proposition~\ref{prop:tree-->hyperbolic}, $G$ acts with only finitely many orbits on the edges of $X$.  Hence $X$ is locally finite and is quasi-isometric to a rough Cayley graph of $G$, which must therefore be hyperbolic.
\qed
\end{proof}
The following conjecture 
asks for a common extension of our Theorem~\ref{MainThm-variant}
and Proposition~36(2) 
in~\cite{direction(aut(tdlcG))}.

\begin{conjecture}
\label{conj:dir(hypGs)} 
Let $G$ be 
a hyperbolic, totally disconnected, locally compact group. 
Then the following holds. 
\begin{enumerate}
\item 
The map 
which assigns 
each element 
of non-trivial scale 
to its attracting boundary point 
defines 
an injection of 
the set of directions 
of~$G$ 
into 
the hyperbolic boundary. 
\item 
Elements 
of~$G$  
with distinct directions 
have pseudo-distance 2. 
\end{enumerate} 
\end{conjecture} 

The next conjecture asks whether the relationship between flat subgroups and the geometry of the rough Cayley graph that may be observed in automorphism groups of trees
or in the setting 
of~\cite{flatrk(AutGs(buildings))} 
holds for hyperbolic totally disconnected, locally compact groups in general. 

\begin{conjecture}
\label{conj:struc(Gs(flatrk1))} 
Suppose that~$G$ 
is 
a hyperbolic, totally disconnected, locally compact group 
which 
does not fix 
a point of 
its hyperbolic boundary. 
Then 
every flat subgroup  
of \flatrk\!~$1$,  
$H$ say,  
of~$G$  
has a limit set 
that contains 
$2$ elements, 
which are both fixed 
by 
$H$. 
%
The group~$H_1$ 
is relatively compact 
and 
is equal to 
the set 
of elliptic elements 
of~$H$.  
\end{conjecture}


%


%
The necessity of the hypothesis that the group should not fix a point on the hyperbolic boundary 
is shown by 
the following example. 

\begin{example}
\label{ex:flatrk1-fixes-single-boundary-point} 
Let $G$ be the semidirect product ${\mathbb Z}\ltimes \mathbb{F}_q((t))$, where $\mathbb{F}_q((t))$ is the ring of formal Laurent series over the finite field $\mathbb{F}_q$ and ${\mathbb Z}$ acts by multiplication by $t$. This group is isomorphic to the group of matrices of the form $\left(\begin{array}{cc}t^n &f \\0 & 1\end{array}\right)$ where $n\in {\mathbb Z}$ and $f\in \mathbb{F}_q((t))$. Then $G$ acts faithfully and co-compactly on the Bruhat-Tits-tree of $SL_2(\mathbb{F}_q((t)))$, a homogeneous tree where every vertex has valency $q+1$, and so is a hyperbolic group by Proposition~\ref{prop:hyperbolicity(cocp-<Gs(Aut(lf-Gromov-hyp-complexes))}. (Although $G$ is not a subgroup of $SL_2(\mathbb{F}_q((t)))$, it does act on this group by conjugation and this action induces an action on the tree.) 

Put $O$ equal to $\mathbb{F}_q((t))$, a compact open subgroup of $G$. Direct calculation shows that, if $g=(n,f)\in G$, then $|gOg^{-1} : gOg^{-1}\cap O|$ is equal to~1 if $n\geq0$ and to $q^{-n}$ if $n<0$ and that these are the minimum possible. Hence $O$ is minimizing for $G$ and $G$ is flat. 

However $G$ does not satisfy the last hypothesis of the conjecture because it fixes a point on the hyperbolic boundary, in this
case the set of ends of the tree. It does not satisfy the conclusions because there is no other end of the tree fixed by $G$. Furthermore, as the above calculation shows, $G_1 =   \mathbb{F}_q((t))$ which is the set of elliptic elements in $G$, and this group 
 is not
compact. 
\end{example}
%


\section{Examples of simple groups of flat rank at most~$1$}\label{sec:examples}

Here 
we list examples 
of \textit{simple}, totally disconnected, locally compact groups 
whose flat rank 
is 
at most~$1$. 
We expect 
none of the listed groups 
to have 
\flatrk\!~$0$. 
Indeed, in any given case 
it is usually easy 
to exhibit 
a hyperbolic element 
of non-trivial scale.

In~\cite{arbre}, 
Tits showed that 
many  
closed subgroups of 
automorphism groups 
of locally finite trees 
are 
simple and provided 
concrete constructions 
of examples 
in terms of 
$a$-coverings. 
As seen in the previous section, these groups 
have \flatrk\! at most~$1$. 

Haglund and Paulin, 
in~\cite{simpl(G-aut(courb-))},  
adapted Tits' methods 
to automorphism groups 
of negatively curved complexes, 
thus providing 
many more 
totally disconnected, locally compact, non-discrete, non-linear, simple groups. 
In order to 
be able 
to formulate 
an analogue 
of Tits' property~(P), 
a central assumption 
in~\cite{arbre}, 
they 
introduced 
an axiomatic framework, 
namely, spaces with walls, 
which allowed them 
to generalize 
Tits' result 
to groups 
acting 
on hyperbolic spaces with walls; 
see Th\'eor\`eme~6.1
in their paper. 

The groups studied by Haglund and Paulin
are non-discrete 
under fairly general conditions; 
compare their Lemme~3.6. 
We suspect that 
all non-discrete groups~$G^+$,  
where~$G$ 
satisfies the conditions 
of~\cite[Th\'eor\`eme~6.1]{simpl(G-aut(courb-))},  
act cocompactly 
on the hyperbolic graph 
associated to 
the space with walls 
in the statement of 
that theorem. 
If so, 
such a group~$G^+$ 
is 
a totally disconnected, locally compact, non-discrete, simple, hyperbolic group 
as a consequence of 
Proposition~\ref{prop:hyperbolicity(cocp-<Gs(Aut(lf-Gromov-hyp-complexes))} and thus has 
\flatrk\! 
at most~$1$ 
by Theorem~\ref{MainThm-variant}.

While 
there is some uncertainty 
as to whether 
the group~$G^+$ 
associated to 
a general group 
satisfying 
the conditions 
of Th\'eor\`eme~6.1 
in Haglund and Paulin's paper 
acts cocompactly, 
all concrete examples 
given in that paper 
do act cocompactly. 
These examples are 
\begin{enumerate} 
\item 
the group 
of type-preserving automorphisms 
of a Bourdon building 
(Th\'eo\-r\`eme~1.1); 
\item 
a subgroup 
of finite index in 
the automorphism group of 
a Benakli-Haglund building 
(Th\'eo\-r\`eme~1.2); 
\item 
a subgroup of 
finite index in 
the automorphism group of 
the Cayley graph of 
a hyperbolic, non-rigid Weyl group 
with respect to 
its standard system of 
generators 
(Th\'eor\`eme~1.3); 
\item 
subgroups 
of finite index in 
the automorphism groups of 
certain 
even polyhedral complexes 
(Th\'eor\`eme~1.4). 
\end{enumerate}
That 
these concrete examples 
do act cocompactly 
is 
no accident. 
It is difficult 
to construct 
complexes with 
uncountable automorphism groups;  
most constructions 
of such complexes 
start from 
a discrete group 
acting cocompactly 
on some complex. 
A notable exception 
is 
the horocyclic product 
of two locally finite trees 
with different valencies; 
while 
the automorphism group 
of this complex 
acts cocompactly,  
there is 
no discrete subgroup 
doing the same 
as shown 
in~\cite{qisos+rig(solvGs)}.

\section{Hyperbolic orbits for groups of automorphisms}
\label{sec:hyperbolic orbits}

The idea of the proof
of Theorem~\ref{thm:altMainThm} 
is straightforward. 
If 
$\mathcal{A}$ contains 
a flat subgroup 
whose rank 
is~$2$, 
then 
there is an orbit of $\mathcal{A}$ 
which contains a subset 
which `looks like' $\mathbb{Z}^2$; 
this is inconsistent 
with the assumption that 
$\mathcal{A}$ has a hyperbolic orbit. 
Since 
we are now dealing with 
spaces 
which are 
not geodesic, 
we must use 
the general definition 
of $\delta$-hyperbolic space 
in terms of 
the Gromov-product;  
recall that 
this states that 
a metric space~$X$  
is $\delta$-hyperbolic 
if and only if 
\[
(x\cdot y)_w\ge \min\{(x\cdot z)_w,(y\cdot z)_w\}-\delta 
\]
for all $w$, $x$, $y$, $z\in X$ 
\cite[III.H.1.20]{GMW319}.  

For our argument,  
we first 
need to 
be able to 
choose the orbit 
at our convenience, 
as the extent to which 
a flat subgroup of \flatrk $2$ 
`looks like' $\mathbb{Z}^2$ 
depends on the orbit. 
The following two results  
take care of that problem. 

\begin{lemma}\label{orbits_quasi-iso}
Let 
a group $\mathcal{A}$ 
act by isometries on a metric space $\mathbf{B}$.  
Then 
any two orbits 
of $\mathcal{A}$ 
in $\mathbf{B}$ 
are $(1,\epsilon)$-quasi-isometric, 
with $\epsilon$ only depending on 
the pair of orbits.   
\end{lemma}
\begin{proof}
Fix two points, 
$M$ and $N$ say, 
in $\mathbf{B}$. 
For every $P\in \mathcal{A}.M$ 
choose an element $\alpha_P\in\mathcal{A}$ 
such that $P=\alpha_P.M$.  
Once this choice is made, 
define a map 
$\tau_{M,N}\colon \mathcal{A}.M\to \mathcal{A}.N$ 
which maps $P$ to $\alpha_P.N$. 
Then, 
for $A$ and $B$ in $\mathcal{A}.M$ 
we have  
\begin{align*}
d(A,B) 
&\leq d(\alpha_A.M,\alpha_A.N)+d(\alpha_A.N,\alpha_B.N)+d(\alpha_B.N,\alpha_B.M)\\
&=d(\alpha_A.N,\alpha_B.N)+2 d(M,N)\,,
\end{align*}
and 
\begin{align*}
d(\alpha_A.N,\alpha_B.N)
&\leq d(\alpha_A.N,\alpha_A.M)+d(\alpha_A.M,\alpha_B.M)+d(\alpha_B.M,\alpha_B.N)\\
&= d(A,B) + 2 d(M,N)
\end{align*}
hence 
$|d(\tau_{M,N}(A),\tau_{M,N}(B))-d(A,B)|\leq 2 d(M,N)$, 
which shows that 
$\tau_{M,N}$ is a $(1,2 d(M,N))$-quasi-isometric embedding. 
Further, 
we have 
\begin{align*}
d(\tau_{M,N}&(\alpha.M),\alpha.N) =
d(\alpha_{\alpha.M}.N,\alpha.N)\\
&\leq  d(\alpha_{\alpha.M}.N,\alpha_{\alpha.M}.M)+
d(\alpha_{\alpha.M}.M,\alpha.M)+
d(\alpha.M,\alpha.N)\\
&=  2d(M,N)\,,
\end{align*}
showing that 
$\tau_{M,N}$ is $2 d(M,N)$-dense. 
\qed 
\end{proof}

\begin{corollary}\label{some-orbit_hyp=all-orbits_hyp}
If one orbit of $\mathcal{A}$ 
in $\bs G.$ is hyperbolic, 
then all are. 
\qed 
\end{corollary}

\begin{lemma}\label{(1,-)-qi-embed preserves_hyp}
Let $f\colon \mathbf{B}_1\to \mathbf{B}_2$ 
be a 
$(1,\epsilon)$-quasi-isometric embedding. 
If $\mathbf{B}_2$ is $\delta$-hyperbolic, 
then $\mathbf{B}_1$ is $(\delta+3\epsilon)$-hyperbolic. 
\end{lemma}
\begin{proof}
From our assumption 
$
|d(x,y) - d(f(x),f(y))|\leq \epsilon
$
for any pair, $x$ and $y$, 
of points in $\mathbf{B}_1$, 
we infer that 
\[
|(x\cdot y)_o - (f(x)\cdot f(y))_{f(o)}|\leq 3\epsilon/2\,. 
\]
We conclude that 
$(x\cdot z)_o-\min\{(x\cdot y)_o,(y\cdot z)_o\}$
differs from 
$$(f(x)\cdot f(z))_{f(o)}-\min\{(f(x)\cdot f(y))_{f(o)},(f(y)\cdot f(z))_{f(o)}\}
$$ 
by at most $3\epsilon$. 
\qed
\end{proof}

Taking 
$f$ in Lemma~\ref{(1,-)-qi-embed preserves_hyp}
to be the inclusion of a subset 
we also obtain 
the following corollary. 

\begin{corollary}\label{subspace(hyp)=hyp}
Every subspace of a hyperbolic space is hyperbolic. 
\qed
\end{corollary}

Corollaries~\ref{some-orbit_hyp=all-orbits_hyp} and~\ref{subspace(hyp)=hyp}  
will prove Theorem~\ref{thm:altMainThm}, 
once 
we show that 
the presence of 
a flat group 
of rank~$r\ge 2$ 
implies 
the existence of 
a subspace 
of~$\bs G.$ 
that is 
an 
integer lattice 
in a normed space 
of dimension~$r$. 
For completeness, 
the proof 
that 
such a lattice 
is not hyperbolic 
is outlined 
after 
the next result.

\begin{lemma}\label{norm>flat-gp}
Let $\mathcal{H}$ be a flat group of automorphisms 
of the totally disconnected locally compact group $G$ with rank~$r$. 
Let $O$ be minimizing for $\mathcal{H}$. 
Then $\left(\mathcal{H}.O,d\right)$ 
is isometric to 
a lattice in 
in a normed 
real linear space 
of dimension~$r$. 
\end{lemma}
\begin{proof}
Since $O$ is minimizing for $\mathcal{H}$, an automorphism $\alpha\in\mathcal{H}$ satisfies $\alpha.O = O$ if and only if $\alpha\in \mathcal{H}_1$ and the map $\alpha\mapsto \alpha.O$ induces a bijection
$$
\mathcal{H}.O \to \mathcal{H}/\mathcal{H}_1.
$$
Composition with the isomorphism $\mathcal{H}/\mathcal{H}_1\to \mathbb{Z}^r$ then produces a bijection $\mathcal{H}.O \to \mathbb{Z}^r$. The metric on $\mathcal{H}.O$ pushes forward to $\mathbb{Z}^r$ via this bijection and the resulting metric on $\mathbb{Z}^r$ is translation-invariant because $\mathcal{H}/\mathcal{H}_1$ and $\mathbb{Z}^r$ are isomorphic as groups. 

An explicit formula may be given
for the distance $d(O,\alpha.O)$ for $\alpha\in \mathcal{H}$.  
There are: a finite set $\Phi=\Phi(\mathcal{H},G)$ 
of surjective homomorphisms
$\rho\colon  {\mathcal H} \to {\mathbb Z}$ 
such that 
the intersection of 
the kernels of elements in $\Phi$ 
equals $\mathcal{H}_1$; 
and a set $\left\{t_\rho\mid \rho\in\Phi\right\}$ of 
integers  greater than one
such that
\[
s_G(\alpha) = \prod_{\rho\in\Phi,\,\rho(\alpha)>0} t_\rho^{\rho(\alpha)}\,,\ (\alpha\in\mathcal{H}),
\]
see~\cite[Theorems 6.12 and 6.14]{tidy<:commAut(tdlcG)}. 
Since $O$ is minimizing for $\mathcal{H}$, 
we further have that
$d(\alpha(O),O)=\log\bigl(s_G(\alpha)\cdot s_G(\alpha^{-1})\bigr)$, 
whence 
\begin{equation*}
\label{eq:normform}
d(\alpha(O),O) = \sum_{\rho\in\Phi} \log(t_\rho)\,|\rho(\alpha)|\,, \ (\alpha\in\mathcal{H}).
\end{equation*}
Composition of each $\rho\in\Phi$ with the isomorphism $\mathcal{H}/\mathcal{H}_1\to \mathbb{Z}^r$ yields  homomorphisms $\tilde\rho : \mathbb{Z}^r \to \mathbb{Z}$ and there
are $\mathbf{z}_\rho\in \mathbb{Z}^r$ such that $\tilde\rho(\mathbf{z}) = \mathbf{z}_\rho.\mathbf{z}$ for each  $\rho\in\Phi$. Hence the bijection $\mathcal{H}.O \to \mathbb{Z}^r$ becomes an isometry if $\mathbb{Z}^r$ is equipped with the translation-invariant metric
\begin{equation}
\label{eq:normform2}
\tilde d(\mathbf{w},\mathbf{z}) := \sum_{\rho\in\Phi} \log(t_\rho)\,| \mathbf{z}_\rho.(\mathbf{w}-\mathbf{z})|\,, \ (\mathbf{w},\,\mathbf{z}\in\mathbb{Z}^r).
\end{equation}
This metric extends to $\mathbb{R}^r$ by the same formula.  
\qed
\end{proof}

We conclude with 
a sketch of 
the argument 
that 
a lattice~$X$   
in a normed space
of dimension~$r\ge 2$ 
is not hyperbolic. 
Given $4$ points ${x}$, ${y}$, ${z}$ and ${w}$ 
in~$X$ 
let $\delta({x},{y},{z})_{w}$ denote 
the quantity 
$ 
\min\{
({y}\cdot {x})_{w},
({x}\cdot {z})_{w} 
\}
-({y}\cdot {z})_{w}
$. 
The number $\delta({x},{y},{z})_{w}$ 
is a lower bound for any $\delta$ such that 
$X$ could be $\delta$-hyperbolic. 
But $\delta(
\lambda{x},\lambda{y},\lambda{z}
)_{\lambda{w}}=
|\lambda|\delta({x},{y},{z})_{w}$ 
for any $\lambda\in\mathbb{Z}$, 
showing that 
no such $\delta$ can exist, 
if we can find a quadruple $({x}, {y}, {z}, {w})$ 
such that  $\delta({x},{y},{z})_{w}$ 
is positive. 
If~$x$ and~$y$ 
are vectors 
for which 
$\|x+y\|+\|x-y\|>\|x\|+\|y\|$ 
then 
$\delta({x},{y},{z})_{\mathbf{0}}$ is positive.  
Such vectors 
exist 
for any 
normed linear space 
of dimension 
at least~$2$; 
vectors 
in the lattice 
can be found 
by rational approximation 
followed by 
scaling by integers.

%




%

%

%

%

%


\end{document}